\documentclass[11pt,a4paper]{article}
\usepackage{amsmath}
\usepackage{amssymb}
\usepackage{amsfonts}
\usepackage{amsthm}
\usepackage{relsize}
\usepackage{setspace}
\usepackage{geometry}
\usepackage{url}
\usepackage{enumerate}

\newtheorem{thm}{Theorem}[section]
\newtheorem{lem}[thm]{Lemma}

\newtheorem{cor}[thm]{Corollary}

\newtheorem*{thma}{Theorem A}
\newtheorem*{thmb}{Theorem B}
\newtheorem*{thmc}{Theorem C}

\theoremstyle{definition}
\newtheorem{defn}[thm]{Definition}
\newtheorem{ex}{Example}

\newcommand{\Aut}{\mathrm{Aut}}

\newcommand{\Sym}{\mathrm{Sym}}

\newcommand{\Out}{\mathrm{Out}}

\newcommand{\Comp}{\mathcal{K}}
\newcommand{\GL}{\mathrm{GL}}
\newcommand{\SL}{\mathrm{SL}}

\newcommand{\ds}{\sigma}
\newcommand{\dm}{\delta}
\newcommand{\pmax}{\lambda}
\newcommand{\cst}{c}

\newcommand{\bF}{\mathbb{F}}
\newcommand{\bN}{\mathbb{N}}
\newcommand{\bP}{\mathbb{P}}

\newcommand{\bZ}{\mathbb{Z}}

\newcommand{\mcK}{\mathcal{K}}

\newcommand{\mcQ}{\mathcal{Q}}

\begin{document}

\title{\Large On finite groups whose Sylow subgroups have a bounded number of generators}

\author{Colin D. Reid\\
Lehrstuhl B f\"{u}r Mathematik\\
RWTH Aachen\\
Templergraben 64, D-52062 Aachen, Germany\\
colin@reidit.net}

\maketitle

\begin{abstract}Let $G$ be a finite non-nilpotent group such that every Sylow subgroup of $G$ is generated by at most $\dm$ elements, and such that $p$ is the largest prime dividing $|G|$.  We show that $G$ has a non-nilpotent image $G/N$, such that $N$ is characteristic and of index bounded by a function of $\dm$ and $p$.  This result will be used to prove that the index of the Frattini subgroup of $G$ is bounded in terms of $\dm$ and $p$.  Upper bounds will be given explicitly for soluble groups.\end{abstract}


\section{Introduction}

\begin{defn}Write $\bP$ for the set of all primes, and $p'$ for the set of all primes except for $p$.

Let $G$ be a finite group.  The \emph{Frattini subgroup} $\Phi(G)$ of $G$ is the intersection of all maximal subgroups of $G$.  The \emph{Fitting subgroup} $F(G)$ of $G$ is the largest nilpotent normal subgroup of $G$.  Note that $\Phi(G) \leq F(G)$ for any finite group; see for instance Theorem 5.2.15 of \cite{Rob}.

Write $d(G)$ for the minimum size of a generating set of $G$.  Define $d_p(G):=d(S_p)$, where $S_p$ is a Sylow $p$-subgroup of $G$; equivalently, $d_p(G) = \log_p|S_p/\Phi(S_p)|$.  Define $\ds(G) := \sum_{p \in \bP} d_p(G)$, and $\dm(G) := \max_{p \in \bP} d_p(G)$.  Note that $\ds(G)$ is bounded by a function of $\dm(G)$ and $\pmax(G)$.

Define $\pmax(G)$ to be the largest prime dividing $|G|$ (set $\pmax(G)=2$ if $G$ is trivial).\end{defn}

Much of the theory of finite groups is built on the study of $p$-subgroups, in particular Sylow subgroups, and the interactions between different primes.  In some situations, a bound on the number of generators of a Sylow subgroup is significant.  For instance, there is a large body of work concerning the ordinary and $p$-modular representation theory of finite groups $G$ for which $d_p(G)=1$, that is, $G$ has a cyclic Sylow $p$-subgroup.  This is of course a special case, as the internal structure of $d$-generator $p$-groups is very much more complex even for $d=2$.  Nevertheless, an arbitrary bound on the number of generators of Sylow subgroups can have surprising consequences.  For instance, it was shown by Guralnick (\cite{Gur}) and Lucchini (\cite{Luc}) that $d(G) \leq \dm(G) + 1$, for any finite group $G$.

This paper concerns results in a similar vein, with the hypothesis that bounds are given for $\ds(G)$ and $\dm(G)$, together with the largest prime $\pmax(G)$.  If $G$ is nilpotent, then there is no bound on the nilpotency class even for $\dm(G) = 2$.  However, there is more control over the manner in which $G$ can fail to be nilpotent.

\begin{thma}Let $G$ be a non-nilpotent finite group.  Let $N$ be a subgroup that is maximal subject to the conditions that $N$ is characteristic and $G/N$ is not nilpotent.  Then exactly one of the following holds, where $\dm = \dm(G)$ and $\pmax = \pmax(G)$.
\vspace{-10pt}
\begin{enumerate}[(i)]  \itemsep0pt
\item The quotient $G/N$ is of the form $S \rtimes H$, where $S$ is an elementary abelian $p$-group for some prime $p$ and $H$ is a nilpotent $p'$-group that acts faithfully on $S$ by conjugation.  The order of $G/N$ is at most $p^{\cst \dm}/2$, where $\cst = \log 288/\log9 < 8/3$.
\item There is a characteristic subgroup of $G/N$ that is the direct product of at most $\dm^2$ copies of a non-abelian finite simple group.  The order of $G/N$ is at most $c^{\dm^2\pmax^2}$, where $c$ is an absolute constant.\end{enumerate}\end{thma}

When combined with Tate's $p$-complement theorem (see Theorem \ref{tate}), Theorem A can be used to obtain a bound on $|G:\Phi(G)|$.

\begin{thmb}Let $G$ be a finite group; set $\dm = \dm(G)$ and $\pmax = \pmax(G)$.
\vspace{-10pt}
\begin{enumerate}[(i)]  \itemsep0pt
\item The index $|G:\Phi(G)|$ is bounded by a function of $\dm$ and $\pmax$.
\item Set $\ds = \ds(G)$.  If $G$ is non-trivial and soluble, then $|G:\Phi(G)| \leq g_{\ds}\lambda^{\ds g_{\ds}}$, where $g_{\ds}$ is obtained as follows:
\[ c= \log 288/\log 9; \quad g_0 = 1 ; \quad g_{i+1} = \pmax^{\cst \dm g_i}, i \geq 0.\]\end{enumerate}\end{thmb}

The following generalisation of Theorem B part (i) to profinite groups strengthens a result of Mel'nikov (\cite{Mel}).

\begin{cor}\label{melcor}Let $G$ be a profinite group, such that every Sylow pro-$p$ subgroup of $G$ is topologically generated by at most $d$ elements, and such that no prime dividing the order of a finite continuous image of $G$ exceeds $n$.  Then $G$ has a pro-(finite nilpotent) open normal subgroup of index bounded by a function of $d$ and $n$.\end{cor}

Finally, it can be concluded from Theorem B that the number of non-abelian composition factors of $G$ is bounded by a function of $\dm(G)$ and $\pmax(G)$.  In fact, a bound on the number of non-abelian composition factors can be obtained with weaker conditions on $G$, using a different method of proof.

\begin{thmc}Let $G$ be a finite group and let $p$ and $q$ be distinct primes.  Then the number of composition factors of $G$ of order divisible by $pq$ is bounded by a function of $(p,q,d_p(G),d_q(G))$.  The total number of non-abelian composition factors of $G$ is bounded by a function of $(d_2(G),d_3(G),d_5(G))$.\end{thmc}

\section{Preliminaries}

\begin{defn}Given a set of primes $\pi$, write $O^\pi(G)$ be the smallest normal subgroup of $G$ such that $G/O^\pi(G)$ is a $\pi$-group, and write $O_\pi(G)$ for the largest normal $\pi$-subgroup of $G$.\end{defn}

Tate's $p$-complement theorem, as given below, plays a critical role in the proofs in this paper.  (Tate actually proved somewhat more, but the form below is sufficient for the purposes of this paper.)

\begin{thm}[Huppert \cite{Hup}; Tate \cite{Tat}]\label{tate}Let $G$ be a finite group, and let $S_p$ be a Sylow $p$-subgroup of $G$ for some prime $p$.  Let $N \unlhd G$ such that $N \cap S_p \leq \Phi(S_p)$ for some prime $p$.  Then $O^p(N)$ is a $p'$-group.  In particular, if $N \cap S_p \leq \Phi(S_p)$ for every $p$ then $N$ is nilpotent.\end{thm}

Here are some useful consequences.

\begin{cor}\label{tatecor}\
\vspace{-10pt}
\begin{enumerate}[(i)]  \itemsep0pt
\item Let $M \unlhd G$ and $N \unlhd G$ such that $M \cap S_p \leq \Phi(S_p)N$ for some prime $p$.  Then $O^p(MN/N)$ is a $p'$-group.  In particular, if $M \cap S_p \leq \Phi(S_p)N$ for every $p$ then $MN/N$ is nilpotent.
\item Let $N \unlhd G$ such that $N \cap S_p \leq \Phi(S_p)$ for every $p$.  Then $F(G/N) = F(G)/N$.
\end{enumerate}\end{cor}

\begin{proof}(i) This follows immediately from the theorem, noting that $S_pN/N$ is a Sylow $p$-subgroup of $G/N$, and that $\Phi(S_pN/N) = \Phi(S_p)N/N$.

(ii) For each prime $p$, let $T_p$ be the lift of the Sylow $p$-subgroup of $F(G/N)$ to $G$; note that $T_p \unlhd G$.  Let $q$ be a prime distinct from $p$.  Then $T_p \cap S_q = N \cap S_q \leq \Phi(S_q)$, so $O^q(T_p)$ is a $q'$-group.  It follows that $O^{p'}(T_p)$ is a $p$-group, and hence the unique Sylow $p$-subgroup of $T_p$.  Now let $T$ be the lift of $F(G/N)$ to $G$.  Then $T$ has a normal Sylow $p$-subgroup for every $p$, namely $O^{p'}(T_p)$, so $T$ is nilpotent, and hence $T \leq F(G)$, in other words $F(G/N) \leq F(G)/N$.  In the other direction, $F(G)/N$ is nilpotent and thus contained in $F(G/N)$.\end{proof}

The bound obtained in case (i) of Theorem A is based on the following:

\begin{thm}[Wolf \cite{Wol}]\label{nilplin}Let $G$ be a nilpotent subgroup of $\GL(n,p^e)$ of order coprime to $p$.  Then $\log_p(2|G|) \leq en \log 32/\log 9$.\end{thm}

For case (ii) of Theorem A, a bound is required on the order and number of the simple subnormal subgroups of $G/N$.  For the order of simple factors, a bound emerges from $\pmax(G)$:

\begin{thm}[Babai, Goodman and Pyber \cite{BGP}: Theorem 5.4 and subsequent remark]  Let $k$ be any positive integer.  Suppose $G$ is a finite simple group whose order has no prime divisor greater than $k$.  Then $|G| < k^{k^2}$, and in fact $|G| < c^{k^2}$ for an absolute constant $c$.\end{thm}

Indeed, given the very restricted size and structure of outer automorphism groups of non-abelian finite simple groups, similar bounds apply to the order of $\Aut(G)$.  Certainly $|\Aut(G)| < |G|^2$ in all cases (see for instance Lemma 2.2 of \cite{Qui}), which leads to the following:

\begin{cor}\label{lambdafsg}  Let $G$ be a non-abelian finite simple group.  Then $|\Aut(G)| < b^{\pmax(G)^2}$ for an absolute constant $b$.\end{cor}

For the number of non-abelian composition factors, on the other hand, the critical constraint for our purposes is the number of generators of the Sylow subgroups.

\begin{defn}Let $\Comp(G)$ be the set of non-abelian simple subnormal subgroups of $G$.  Given a set of primes $\pi$, let $\Comp_\pi(G)$ consist of those $Q \in \Comp(G)$ such that $p$ divides $|Q|$ for all $p \in \pi$, and let $E_\pi(G) = \langle \Comp_\pi(G) \rangle$.  Note that $E_\pi(G)$ is the direct product of the elements of $\Comp_\pi(G)$.\end{defn}

\begin{lem}\label{comporb}Let $\pi$ be a set of primes containing $p$.  Let $n$ be the number of orbits of $S_p$ acting on $\mcK$ by conjugation.  Then $n \leq d_p(G)$.\end{lem}

\begin{proof}We may assume that $G = E_\pi(G)S_p$.  Since $|S_p:\Phi(S_p)|=p^{d_p(G)}$, there is a subset $\mcQ$ of $\Comp_\pi(G)$ such that $|\mcQ| \leq d_p(G)$ and such that $E_\pi(G) \cap S_p \leq (R \cap S_p)\Phi(S_p)$, where $R = \langle \mcQ \rangle$. Let $N$ be the normal closure of $R$ in $G$; clearly $N \leq E_\pi(G)$.  Then $O^p(E_\pi(G)/N)$ is a $p'$-group by Corollary \ref{tatecor}, and hence $E_\pi(G) = N$.  It follows that given $Q \in \Comp_\pi(G)$, there is some $x \in G$ such that $Q^x \in \mcQ$; since $E_\pi(G)$ normalises $Q$, there is a suitable $x$ in $S_p$.  Hence $n \leq |\mcQ| \leq d_p(G)$ as required.\end{proof}

We will also need to use the fact that a nilpotent permutation group has order that is bounded by an exponential function of the degree:

\begin{thm}[Vdovin \cite{Vdo}]\label{nilper}Let $N$ be a nilpotent subgroup of $\Sym(n)$ of largest possible order.  If $n = 2(2k+1)+1$ for some $k \in \bZ$, then $N$ is isomorphic to the direct product of a Sylow $2$-subgroup of $\Sym(n-3)$ with a cyclic group of order $3$; otherwise, $N$ is a Sylow $2$-subgroup of $\Sym(n)$.  In particular, the order of a nilpotent subgroup of $\Sym(n)$ is at most $2^n$.\end{thm}

\section{The main theorems}

\begin{proof}[Proof of Theorem A]We may assume that $G = G/N$, so that for every non-trivial characteristic subgroup $K$ of $G$, the image $G/K$ is nilpotent.  Let $M$ be a minimal characteristic subgroup of $G$.

Suppose that $M$ is soluble.  Then $M$ is an elementary abelian $p$-group for some $p$.  It follows that $G$ has a normal and hence characteristic Sylow $p$-subgroup $S$ say, as $G/M$ is nilpotent.  By the Schur-Zassenhaus theorem, $G$ is of the form $S \rtimes H$, where $H$ is a $p'$-group; note that $H \cong G/S$, so $H$ is nilpotent.  Furthermore, $F(G/\Phi(S)) = F(G)/\Phi(S) < G/\Phi(S)$ by Corollary \ref{tatecor}, so $G/\Phi(S)$ is not nilpotent; thus $\Phi(S) = 1$.  So $S$ is elementary abelian, and $\log_p(|S|) = d(S) \leq \dm$.

Note that $C_S(H) = O_{p'}(G)$; moreover, $F(G) = SO_{p'}(G) < G$, so $O_{p'}(G)$ does not contain $H$.  This ensures $G/O_{p'}(G)$ is non-nilpotent, so $O_{p'}(G) = 1$.  Thus $H$ acts faithfully on $S$ by conjugation.  It follows that $H$ is isomorphic to a nilpotent $p'$-subgroup of $\GL(\dm,p)$; hence $|H| \leq p^{(\cst-1)\dm}/2$ by Theorem \ref{nilplin}, so $|G| = |H||S| \leq p^{\cst\dm}/2$.  We have now proved all of the assertions for case (i).

Now suppose that $M$ is insoluble.  Then $M$ is a direct product of a set $\Omega$ of isomorphic copies of a non-abelian finite simple group $Q$.  Let $G$ act on $\Omega$ by conjugation and let $p$ and $q$ be two distinct primes which divide the order of $Q$.  Then the number of orbits of $S_p$ on $\Omega$ is at most $\dm$ by Lemma \ref{comporb}; moreover, the orbits all have the same size, say $p^m$, since $M$ acts trivially on $\Omega$ and $S_pM$ is characteristic in $G$.  Thus $|\Omega| = xp^m$ for some non-negative integers $x$ and $m$ such that $x \leq \dm$; similarly $|\Omega| = x' q^{m'}$ for some $x'$ and $m'$ such that $x' \leq \dm$.  This produces an upper bound for $|\Omega|$ as follows:
\[ |\Omega| = \gcd(|\Omega|,|\Omega|) \leq \gcd(xp^m,x') \gcd(xp^m,q^{m'}) = \gcd(xp^m,x') \gcd(x,q^{m'}) \leq x'x \leq \dm^2.\]

Let $R = \bigcap_i N_G(Q_i)$.  Then $G/R$ is isomorphic to a nilpotent subgroup of $\Sym(\Omega)$, so $|G/R| \leq 2^{\dm^2}$ by Theorem \ref{nilper}.

By Corollary \ref{lambdafsg} we have $|\Aut(Q)| < b^{\pmax^2}$ where $b$ is an absolute constant.  Since $G/C_G(M)$ is insoluble and hence non-nilpotent, we also have $C_G(M) = 1$, so that $|R| \leq |\Aut(Q)|^{\dm^2}$.  Hence $|G| \leq (2b^{\pmax^2})^{\dm^2} < (2b)^{\dm^2\pmax^2}$.    We have now proved all of the assertions for case (ii).\end{proof}

An obstacle to improving the bound in case (ii) is illustrated by the following example.

\begin{ex}(My thanks go to Robert Wilson for suggesting this example.)  Let $H = \SL(2,q)$, where $q = p^{p^e}$, and adjoin the Frobenius automorphism $f: x \mapsto x^p$ of the defining field $\bF_q$ to form the semidirect product $G = H \rtimes \langle f \rangle$.  The order of $f$ is $p^e$, and a Sylow $p$-subgroup of $G$ is generated by $f$ together with the element $s$ of $H$, where
\[ s = \left( \begin{array}{cc}
1 & \mu \\
0 & 1 \end{array} \right)\]
such that $\mu$ is a primitive element of $\bF_q$.  There is also a cyclic subgroup of $H$ of order $(q^2-1)$; since $|G|/(q^2-1)$ is a power of $p$, it follows that all Sylow $r$-subgroups of $G$ are cyclic for $r \not= p$.

Thus $d_p(G) =2$, and $d_r(G)=1$ for any prime $r$ not equal to $p$.  So the Sylow subgroups of $G$ have a bounded number of generators, for any values of $p$ and $e$.  On the other hand, given a non-nilpotent image $G/N$ of $G$, then $N$ is a proper normal subgroup of $\SL(2,q)$, and hence $|N| \leq 2$.  In other words, every non-nilpotent image of $G$ has order at least $p^e q(q^2-1)/2 = p^{p^e + e}(p^{2p^e} - 1)/2$.\end{ex}

The non-nilpotent images described in Theorem A will now be used to prove Theorem B and its application to profinite groups.

\begin{proof}[Proof of Theorem B]Choose a sequence $G_i$ of subgroups as follows: $G_0=G$, and if $G_i$ is not nilpotent, then $G_{i+1}$ is a subgroup of $G_i$ that is maximal subject to the conditions that $G_{i+1}$ is characteristic in $G_i$ and $G_i/G_{i+1}$ is non-nilpotent.  The sequence terminates with a nilpotent characteristic subgroup $G_t$ for some $t \geq 0$.

Let $r_i = \sum_{p \in \bP} \log_p|(G_i \cap S_p)\Phi(S_p)/\Phi(S_p)|$.  Then $r_0 = \ds$, and $r_i \geq r_{i+1}$ for all $i$.  Suppose $r_i = r_{i+1}$ for some $i < t$.  Then $(G_i \cap S_p)\Phi(S_p) = (G_{i+1} \cap S_p)\Phi(S_p)$ for all $p \in \bP$, so $G_i/G_{i+1}$ is nilpotent by Corollary \ref{tatecor}, contradicting the choice of $G_{i+1}$.  Hence $r_0, \dots, r_t$ is a strictly decreasing sequence of non-negative integers, bounded above by $\ds$; thus $t \leq \ds$.

Now $\Phi(G) \geq \Phi(G_t)$ as $G_t$ is normal in $G$, so $|G:\Phi(G)| \leq |G_t:\Phi(G_t)|$.  As $G_t$ is nilpotent, the quotient $G_t/\Phi(G_t)$ is abelian of square-free exponent, so $|G:G_t| \leq \pmax^{\ds(G_t)}$; in turn, $\ds(G_t) \leq |G:G_t|(\ds-1)+ 1$ by the Schreier index formula.  Thus a bound for $|G:\Phi(G)|$ in terms of $\dm$ and $\pmax$ will follow from a similar bound for the index of $G_t$.

Suppose $|G:G_i|$ is bounded by a function of $\dm$ and $\pmax$ for some integer $i$.  Then $|G_i:G_{i+1}|$ is bounded by a function of $\pmax(G_i)$ and $\dm(G_i)$ by Theorem A; as $\pmax(G_i) \leq \pmax$ and $\dm(G_i) \leq |G:G_i|\dm$, it follows that $|G:G_{i+1}|$ is bounded by a function of $\dm$ and $\pmax$.  Hence $|G:G_j|$ is bounded by a function of $\dm$ and $\pmax$ for any fixed $j$, and thus $|G:G_t|$ is bounded by a function of $\dm$ and $\pmax$, giving a bound for $|G:\Phi(G)|$ as required for (i).

Now assume $G$ is non-trivial and soluble.  Suppose $|G:G_i| \leq g_i$ for some integer $i$.  Then $\dm(G_i) \leq \dm g_i - g_i + 1$ by the Schreier index formula.  It follows from Theorem A that the index $|G:G_{i+1}|$ is at most $g_i\pmax^{\cst (\dm g_i - g_i +1)}/2$.  Thus
\[ |G:G_{i+1}| \leq g_i\pmax^{\cst (\dm g_i - g_i +1)}/2 = \pmax^{\cst \dm g_i}g_i/(2\pmax^{\cst (g_i-1)}) \leq \pmax^{\cst \dm g_i} = g_{i+1}.\]
Thus $|G:G_t|$ is at most $g_t$ by induction, which is at most $g_{\ds}$.  Hence $|G:\Phi(G)|$ is at most $g_{\ds}\pmax^{g_{\ds}(\ds-1)+ 1}$, which is at most $g_{\ds}\pmax^{\ds g_{\ds}}$ as required for (ii).\end{proof}

\begin{proof}[Proof of Corollary \ref{melcor}] Let $K$ be an open normal subgroup of $G$.  The hypotheses ensure that $\dm(G/K) \leq d$ and $\pmax(G/K) \leq n$; thus by Theorem B, $|G/K:\Phi(G/K)| \leq f(d,n)$ for some function $f$.  Let $N_K$ be the lift of $\Phi(G/K)$ to $G$, and let $N$ be the intersection of the $N_K$ as $K$ ranges over all open normal subgroups.  Then $N$ is closed in $G$; moreover, $N$ is the inverse limit of finite nilpotent groups, while $G/N$ is the inverse limit of groups of order at most $f(d,n)$, so $G/N$ itself has order at most $f(d,n)$.  Hence $N$ is an open pro-(finite nilpotent) normal subgroup of $G$ of index at most $f(d,n)$.\end{proof}

\section{A bound on the number of non-abelian composition factors}

To obtain a bound on the number of non-abelian composition factors, with weaker hypotheses than required for Theorem B, we will make use of some more known results.  Theorems \ref{simordthm} and \ref{schconj} state some facts about finite simple groups that follow from the classification, while Theorem \ref{senstr} is a result from the theory of arithmetic.

\begin{thm}[strengthening of the Odd Order Theorem (\cite{Fei})]\label{simordthm}Let $G$ be a non-abelian finite simple group.  Then $|G|$ is divisible by at least one of $6$ and $10$.\end{thm}

\begin{thm}[Schreier conjecture]\label{schconj}Let $G$ be a finite simple group.  Then $\Out(G)$ is soluble.\end{thm}

\begin{defn}Given $x \in \bN$, let $s_p(x)$ be the sum of the digits of the base-$p$ expansion of $x$.\end{defn}

\begin{thm}[Senge, Straus \cite{Sen}]\label{senstr}Let $p$ and $q$ be distinct primes and let $x \in \bN$.  Then $x$ is bounded by a function of $(p,q,s_p(x),s_q(x))$.\end{thm}
 
\begin{proof}[Proof of Theorem C]Let $\pi = \{p,q\}$ and let $x=|\Comp_\pi(G)|$.  We may assume that every non-trivial normal subgroup of $G$ has a composition factor of order divisible by $pq$.  This ensures that $G$ acts faithfully on $E_\pi(G)$.  Let $S_r$ be a Sylow $r$-subgroup of $G$ where $r$ is $p$ or $q$.  Then by Lemma \ref{comporb}, $S_p$ has at most $d_p(G)$ orbits on $\Comp_\pi(G)$; since each orbit has $p$-power order, this forces $s_p(x) \leq d_p(G)$.  Similarly, $s_q(x) \leq d_q(G)$.  It follows from Theorem \ref{senstr} that $x$ is bounded by a function of $(p,q,d_p(G),d_q(G))$.

Consider the quotient $G/E_\pi(G)$; note that this is isomorphic to a subgroup of $\Out(E_\pi(G))$.  Let $N= \bigcap \{ N_G(Q) \mid Q \in \Comp_\pi(G)\}$.  Then $|G/N| \leq x!$ ; also, $N/E_\pi(G)$ is soluble by Theorem \ref{schconj}.  It follows that in some (and hence any) composition series for $G$, the number of factors of order divisible by $pq$ is bounded by a function of $x$, and hence by a function of $(p,q,d_p(G),d_q(G))$.

The final assertion now follows immediately from Theorem \ref{simordthm}.\end{proof}

\section{Acknowledgments}This paper is partly based on results obtained by the author while under the supervision of Robert Wilson at Queen Mary, University of London (QMUL).  My thanks go to Charles Leedham-Green, L\'{a}szl\'{o} Pyber and Robert Wilson for their comments and advice, and to EPSRC and QMUL for their financial support.

\end{document}